\documentclass[12pt]{elsarticle}

\usepackage{amssymb}
\usepackage{amsthm}
\usepackage{amsmath}
\usepackage{xcolor}

\newtheorem{theorem}{Theorem}
\newtheorem{example}[theorem]{Example}
\newtheorem{lemma}[theorem]{Lemma}
\newtheorem{remark}[theorem]{Remark}

\begin{document}

\begin{frontmatter}



\title{ An explicit univariate and radical parametrization of the septic proper Zolotarev polynomials in power form}


\author[1]{Heinz-Joachim Rack}
\author[2]{Robert Vajda\corref{cor1} \fnref{fn1}}
\cortext[cor1]{Corresponding author}
\fntext[fn1]{Supported by the EU-funded Hungarian grant 
EFOP-3.6.1-16-2016-00008 and the grant by the Ministry of Human Capacities TUDFO/47138-1/2019-ITM.}

\address[1]{Dr. Rack Consulting GmbH, Steubenstrasse 26 a, 58097 Hagen, Germany}
\address[2]{Bolyai Institute, University of Szeged, Aradi Vertanuk tere 1, 6720 Szeged, Hungary}

\begin{abstract}
The problem of determining an explicit one-parameter power form representation of the proper $n$-th degree 
Zolotarev polynomials on $[-1,1]$ can be traced back to P. L. Chebyshev, see \cite{Zolotarev1877}.
It turned out to be complicated, even for 
small values of $n$. Such a representation was known to A. A. Markov (1889) \cite{AAMarkov1889} for $n=2$ and $n=3$, see also \cite{Carlson83}. \emph{But already
for $n=4$ it seems that nobody really believed that an explicit form can be found. 
As a matter of fact it was, by V. A. Markov in 1892} \cite{VAMarkov}, as A. Shadrin put it in 2004 \cite{Shadrin04}, see also \cite{Rack89}, \cite{Rack17a}. 
The next higher degrees, $n=5$ and $n=6$, were resolved only recently, by G. Grasegger and N. Th. Vo (2017) \cite{Grasegger17} respectively by the present authors (2019) \cite{RackVajda19}. 
In this paper we settle the case $n=7$ using symbolic computation.
The parametrization for the degrees $n\in \{2,3,4\}$ is a rational one, whereas for $n\in \{5,6,7\}$ it is a radical one.
However, the case $n=7$ among the radical parametrizations requires special attention, since it is not a simple radical one.
\end{abstract}


\end{frontmatter}

\section{Introduction and historical remarks}\label{sectintro}

With $||.||_{\infty}$ denoting the uniform norm on $\boldsymbol{I}=[-1,1]\subset\mathbb{R}$, Chebyshev \cite{Chebyshev}
found that
\begin{equation}\label{genChebext} 
\min_{(a_{0,n},\cdots,a_{n-1,n)}\in \mathbb{R}^n}||\widetilde{P}_n||_{\infty}=2^{1-n},\hbox{ where } \widetilde{P}_{n}(x)=\sum_{k=0}^{n-1} a_{k,n}\, x^k+x^n.
\end{equation}
The least possible value $2^{1-n}$ is attained if $\widetilde{P}_n(x)={\widetilde{P}}^*_n(x)=2^{1-n}T_n(x)=\sum_{k=0}^{n-1} a_{k,n}^{*}\, x^k+x^n$
with known optimal coefficients $a_{k,n}^{*}$. 
Here $T_n$ with $||T_n||_{\infty}=1$ denotes the $n$-th 
(normalized) Chebyshev polynomial of the first kind with respect to $\boldsymbol{I}$, see \cite[p. 6, p. 67]{Rivlin74} or \cite[p. 384]{Milovanovic} 
for details.
In 1867 Chebyshev himself proposed to his student E. I. Zolotarev, see \cite[p. 2]{Zolotarev1868}, 
an extension of (\ref{genChebext}) by requiring that not only the first but also
the second leading coefficient, $a_{n-1,n}$, is to be kept fixed. This extension was later re-named 
as Zolotarev's first problem (ZFP) and amounts, for a given $n\ge2$, to the 
determination of 
\begin{align}\label{genZFP}
\min_{(a_{0,n},\cdots,a_{n-2,n)}\in \mathbb{R}^{n-1}}||\widetilde{Z}_{n,s}||_{\infty}=L_n(s),\, \hbox{where}\\ 
\widetilde{Z}_{n,s}(x)=\sum_{k=0}^{n-2} a_{k,n}\, x^k+(-ns )x^{n-1}+x^n, \nonumber
\end{align} 
and of the extremal polynomial, $\widetilde{Z}^{*}_{n,s}$, where $s\in \mathbb{R}$ is assumed.
Thus $a_{n,n}=1$ and the second leading coefficient, $a_{n-1,n}=(-ns)$, although thought of as being fixed, 
may attain arbitrary values, 
so that we save the notation $s_0$ for a concrete prescribed number $s$.
Correspondingly, we shall then write $L_{n}(s_0)$ for a concrete minimum in (\ref{genZFP}) and
$\widetilde{Z}^{*}_{n,s_0}$ for a concrete extremal (minimal) polynomial. 
It is well known that one may restrict the range of $s$ to $s>0$, and that, for $0< s\le\tan^2(\frac{\pi}{2n})$,
$\widetilde{Z}^{*}_{n,s}$ is given by a distorted $\widetilde{P}^{*}_n$, see e.g.
\cite[p. 16]{Achieser98}, \cite[p. 57]{Achieser03}, \cite{Carlson83}, \cite[p. 405]{Milovanovic} 
for details, and is called a \emph{monic improper Zolotarev polynomial}.

For the range $s>\tan^2(\frac{\pi}{2n})$ however, on which we focus in this paper, the solution $\widetilde{Z}^{*}_{n,s}$ of
ZFP is considered as \emph{very complicated}, see e.g. \cite[p. 27]{Achieser98}, \cite{Carlson83}, \cite[p. 407]{Milovanovic},
\cite{Peherstorfer91}, as \emph{unwieldy} \cite[p. 118]{Tikhomirov}, or even as \emph{mysterious} \cite{Todd84}, and is called a \emph{monic proper} \cite{Shadrin04}, 
\cite[p. 160]{Voronovskaja70} or \emph{hard-core} \cite[p. 407]{Milovanovic}, \cite{Rivlin75} \emph{Zolotarev polynomial}.
The min-max-problems in (\ref{genChebext}) and (\ref{genZFP}) can be viewed as problems of best uniform approximation to $x^n$ respectively 
$(-n s)x^{n-1}+x^n$ by polynomials of degree $n-1$ respectively $n-2$.

Zolotarev provided $\widetilde{Z}^{*}_{n,s}$ in 1868 \cite{Zolotarev1868} and in a reworked form in 1877 \cite{Zolotarev1877}, however not,
as is suggested by the task (\ref{genZFP}), in an algebraic power form with explicit optimal coefficients.
Rather, he presented $\widetilde{Z}^{*}_{n,s}$ in terms of elliptic integrals and functions, see also e.g. 
\cite[p. 18]{Achieser98}, \cite[p. 280]{Achieser03}, \cite{Carlson83}, \cite{Erdos42}, \cite{Lebedev94}, \cite[p. 407]{Milovanovic}, \cite{Peherstorfer04} for details.
A. A. Markov \cite[p. 264]{AAMarkov1906} expressed his reservation about Zolotarev's solution: 
\emph{Being based on the application of elliptic functions, Zolotarev's solution is too complicated to be useful in practice}. 
As it is expounded in \cite[Section 3]{Carlson83}, to deduce from Zolotarev's elliptic solution an algebraic power form solution 
(sometimes called \emph{synthesizing} \cite[p. 1066]{Chen}) turns out to be \emph{unexpectedly complicated}, even for the first reasonable polynomial degree $n=2$.
As F. Peherstorfer put it in 2006 \cite[p. 143]{Peherstorfer06}, \emph{there was and still is a demand for a description [of proper Zolotarev polynomials]
without elliptic functions}. 
In literature there are scattered several approaches to solve ZFP algebraically and thus to avoid the use of elliptic functions, see \cite[Section 1]{RackVajda19}.

But when it comes down to represent the monic proper Zolotarev polynomial $\widetilde{Z}^{*}_{n,s}$, or its normalized version
$Z^{*}_{n,s}=\widetilde{Z}^{*}_{n,s}/L_n(s)$ (with $||Z^{*}_{n,s}||_{\infty}=1$), 
as a polynomial of degree $n$ in power form with explicit parameterized coefficients, then such solutions are known only
for $2\le n\le6$, see Section \ref{sec26} below. 
Upon using symbolic computation as implemented in \emph{Maple} \cite{maple} and \emph{Mathematica} \cite{MMA1}, 
we are now able to provide such an explicit power form solution 
even for the degree $n=7$, see Section \ref{sec7} below. This contributes to the solution
of ZFP which is one of E. Kaltofen's \emph{favorite open problems} in symbolic computation \cite[Section 2]{Kaltofen}. 
In the conference paper \cite{Lazard} it is claimed to have solved ZFP by symbolic computation even for $6\le n\le12$,
but actually a theoretical solution strategy is delineated, without providing a solution formula or a concrete example, and in particular without representing the extremal polynomial in a parameterized power form for a given $n$.
For an algorithm-based algebraic solution formula to ZFP for $6\le n\le12$ see \cite{Rack18}.

\section{Explicit analytical one-parameter power form representation of the proper 
   Zolotarev polynomials of degree $2\le n\le6$}\label{sec26}
   
With the goal to find a convenient parametrization for the coefficients of the extremal polynomial in (\ref{genZFP}) 
with a parameter $t\in I_n$ from some \emph{finite} open parameter interval
$I_n\subset\mathbb{R}$ $(n>2)$, and following the literature, see \cite{Grasegger17}, \cite[Secton 14]{Paszkowski62},
\cite{Rack89}, \cite{RackVajda19}, \cite[Section 1.4]{Shadrin04}, 
  we now change our notation and strive to obtain the solution of ZFP in the form   
\begin{equation}\label{mzolpargen}
\widetilde{Z}_{n,t}(x)=\sum_{k=0}^{n-2} a_{k,n}(t)\, x^k+(-ns(t) )x^{n-1}+x^n,
\end{equation}
where the explicit coefficients $a_{k,n}(t)$ and $-n s(t)$ depend injectively on $t\in I_n$. The least
deviation (on $\boldsymbol{I}$) of $\widetilde{Z}_{n,t}$ from the zero-function is $||\widetilde{Z}_{n,t}||_{\infty}=L_n(t)$.
For a prescribed $t=t_0\in I_n$ there holds $\min_{(a_{0,n},\dots,a_{n-2,n})\in \mathbb{R}^{n-1}}||P_n||_{\infty}$ $=L_n(t_0)$, 
where $P_{n}(x)=\sum_{k=0}^{n-2} a_{k,n} x^k+(-ns(t_0))x^{n-1}+x^n$.
Thus, for a given fixed degree $n$, (\ref{mzolpargen}) represents an infinite family of $n$-th degree 
monic proper Zolotarev polynomials. For a prescribed $s=s_0\in (\tan^2\left(\frac{\pi}{2n}\right),\infty)$
one then has to solve the equation $s(t)=s_0$ for $t$ and to insert the unique solution
$t=t_0\in I_n$ into $\widetilde{Z}_{n,t}(x)$ in order to get the desired solution $\widetilde{Z}^{*}_{n,s_0}(x)=\widetilde{Z}_{n,t_0}(x)$ in (\ref{genZFP})
for the given $s=s_0$.

Analogously, we will denote by
\begin{equation}\label{nzolpargen}
Z_{n,t}(x)=\sum_{k=0}^{n} b_{k,n}(t) x^k\hbox{ with } t\in I_n \hbox{ and } b_{n,n}(t)\neq0
\end{equation}
the \emph{normalized} proper Zolotarev polynomials $Z_{n,t}=\widetilde{Z}_{n,t}/L_n(t)$ with $||Z_{n,t}||_{\infty}$$=1$.  
For a prescribed $s=s_0\in (\tan^2\left(\frac{\pi}{2n}\right),\infty)$
one then has to equate $b_{n-1,n}(t)/$ $b_{n,n}(t)$ with $(-n s_0)$ and to solve for $t$,
and finally to insert the unique solution
$t=t_0\in I_n$ into $Z_{n,t}(x)/b_{n,n}(t)$ in order to get the desired solution $\widetilde{Z}^{*}_{n,s_0}(x)=Z_{n,t_0}(x)/b_{n,n}(t_0)$ 
in (\ref{genZFP}) for the given $s=s_0$.

So the key question is: How to choose the parameter intervals $I_n$ and the parameterized coefficients
$a_{k,n}(t)$ and $-ns(t)$ in (\ref{mzolpargen}) respectively $b_{k,n}(t)$ in (\ref{nzolpargen})?
Before providing an answer for $n=7$ we first allude to known solutions $Z_{n,t}$
(possibly after some rearrangement) for the polynomial degrees $2\le n\le 6$:
  
For $n=2$ and $n=3$ see \cite{Carlson83}, \cite{Grasegger17}, \cite[p. 246]{Lebedev94}, \cite{AAMarkov1889}, \cite[p. 156]{Paszkowski62}, 
\cite{Rack17b}, \cite[p. 98]{Voronovskaja70}.
  
For $n=4$ see \cite{Grasegger17}, \cite[p. 246]{Lebedev94}, \cite[p. 73]{VAMarkov}, \cite{Rack89}, \cite{Rack17a}, \cite{Shadrin04},
and note the remarkable comment by Shadrin \cite[Section 1.4]{Shadrin04} as quoted in the Abstract of the present paper.
  
Still in 2014 Shadrin \cite[p. 1185]{Shadrin14} was right in writing that \emph{there is no explicit expression for [normalized proper] 
Zolotarev polynomials of degree $n>4$}. But already in 2017 Grasegger \& Vo \cite{Grasegger17} provided such an explicit expression 
for the degree $n=5$, see also \cite{Collins96}, \cite{Malyshev02}, \cite{Rack17b}.

Only two years later the present authors provided such an explicit expression for the degree $n=6$, 
see \cite{RackVajda19}. It goes without saying that the complexity of the explicit expressions (\ref{mzolpargen}) and (\ref{nzolpargen}) 
increases dramatically when the degree $n$ grows,
see also corresponding remarks in \cite[p. 511]{Bernstein1952}, \cite[p. 21]{Bogatyrev02}, \cite[p. 932]{Malyshev02}. 
A further complication creeps in due to the fact that the parametrization in (\ref{mzolpargen}) and 
(\ref{nzolpargen}) is a radical one for $n\in\{5,6,7\}$
(obtained by computer-aided symbolic computation), whereas it is a rational one for $n\in\{2,3,4\}$
(obtained by pencil and paper). This may explain the time gap of 125 years between the parameterized solution for
$n=4$ in \cite{VAMarkov} and the parameterized solution for $n=5$ in \cite{Grasegger17}.
Furthermore, the case $n=7$ among the radical parametrizations is exceptional and hence requires a special treatment, see \cite[p. 179]{Grasegger17} and Section \ref{comppath} below.

It follows from Approximation Theory, see \cite{Achieser98}, \cite[p. 280]{Achieser03}, \cite{Carlson83}, \cite{Erdos42},
\cite[p. 243]{Lebedev94}, \cite[p. 404]{Milovanovic}, \cite[p. 67]{Peherstorfer99} that on the solution
$\widetilde{Z}_{n,t}$ in (\ref{mzolpargen}) 
there can be imposed, without loss of generality, certain definite conditions: There must exist $n$ equioscillation points
$-1=z_0(t)$ $<z_1(t)<\dots<z_{n-2}(t)<z_{n-1}(t)=1$ on $\boldsymbol{I}$, where $\widetilde{Z}_{n,t}$ 
attains the values $\pm L_n(t)$ 
alternately, and its first derivative vanishes at the interior equioscillation points. One may assume that at
$-1=z_0(t)$ the value $(-1)^n L_n(t)$ and hence at $z_{n-1}(t)=1$ the value $-L_n(t)$ is attained.
Furthermore, there exists an interval $[\alpha(t),\beta(t)]$ to the right of $\boldsymbol{I}$ whose endpoints are
also equioscillation points of $\widetilde{Z}_{n,t}$ (with value $-L_n(t)$ at $\alpha(t)$ and value $L_n(t)$ at $\beta(t)$),
and there exists a point $\gamma(t)=(\alpha(t)+\beta(t))/2-s(t)$ with $1<\gamma(t)<\alpha(t)<\beta(t)$ where the first derivative of $\widetilde{Z}_{n,t}$
vanishes. In addition, the uniform norm of  $\widetilde{Z}_{n,t}$ on $\boldsymbol{I}\cup [\alpha(t),\beta(t)]$ 
must be $L_n(t)$, and  $\widetilde{Z}_{n,t}$ 
must satisfy the Abel-Pell differential equation \cite[p. 17]{Achieser98}
\begin{equation}\label{AbelPellEquation}
\frac{(1-x^2)(x-\alpha(t))(x-\beta(t))(\widetilde{Z}_{n,t}\,'(x))^2}{n^2(x-\gamma(t))^2}+(\widetilde{Z}_{n,t}(x))^2=(L_n(t))^2,
\end{equation}
and the points $z_1(t),\dots, z_{n-2}(t)$, $\alpha(t)$, $\beta(t)$ 
must satisfy the Peherstorfer-Schiefermayr system of nonlinear equations \cite[p. 68]{Peherstorfer99}
\begin{align}
\alpha(t)+\beta(t)+2\sum_{j=1}^{n-2}z_{j}(t)-2n s(t)=0 \label{PSE}\\
(-1)^k+2\sum_{j=1}^{n-2}(-1)^j (z_{j}(t))^k+(-1)^{n-1}(1+(\alpha(t))^k-(\beta(t))^k)=0\label{PSE2}\\
\hbox{for } k=1,\dots,n-1.\nonumber
\end{align}

Analogous conditions can be imposed on $Z_{n,t}$ with $L_n(t)$ being replaced by $1$. Note that in literature also the polynomials
$-\widetilde{Z}_{n,t}$ and $-Z_{n,t}$ go by the name of monic respectively normalized proper Zolotarev polynomials.
For abbreviation we henceforth set
$\alpha(t)=\alpha$, $\beta(t)=\beta$, $\gamma(t)=\gamma$, $z_{k}(t)=z_k$ for $k=0,\dots,n-1$.
Following \cite{MMA1} we denote by ${\rm Root[}p(x),k{\rm]}$ the $k$-th root of the polynomial equation $p(x)=0$.


\section{Explicit analytical one-parameter power form representation of the normalized proper 
    Zolotarev polynomials of degree $n=7$}\label{sec7}

Our main result is the representation of the family of normalized proper Zolotarev polynomials $Z_{7,t}$ of degree $7$ 
in the parameterized power form (\ref{nzolpargen}), which implies the representation of $\widetilde{Z}_{7,t}$ in the form (\ref{mzolpargen}).  
Moreover, we provide the two normalized improper Zolotarev polynomials to which $Z_{7,t}$ transforms when the parameter $t$ tends towards the boundaries of $I_7$. The symbolic computations which have led us to Theorem \ref{mainthm} is demonstrated in Section \ref{comppath}.
 
\begin{theorem}\label{mainthm}
Let $t\in I_7=(\theta,-13)$, where $\theta={\rm Root[}-1847-93x+27x^2+x^3,1{\rm]}=-27.963755\ldots$, $\omega=\omega(t)=\sqrt{3 (506 + 75 t - t^3)}$, and
$q=q(t)=2^{14}3^3(1+t)^{12}$.
Then the septic normalized proper Zolotarev polynomial can be parametrized as follows: 
\begin{align}\label{mainthmformula}
&Z_{7,t}(x)=\sum_{k=0}^{7} b_{k,7} (t) x^k=
\frac{1}{q}\times\nonumber\\
&\big((p_{01}+\omega p_{02}) - \sqrt{p_{11}+\omega p_{12}}x+(p_{21}+\omega p_{22})x^2+\sqrt{p_{31}+\omega p_{32}}x^3+\\
&(p_{41}+\omega p_{42}) x^4- \sqrt{p_{51}+\omega p_{52}}x^5+(p_{61}+\omega p_{62})x^6+\sqrt{p_{71}+\omega p_{72}}x^7\big),\nonumber
\end{align}
where the polynomials $p_{k1}=p_{k1}(t)$ and $p_{k2}=p_{k2}(t)$, $0\le k\le7$, are given below:

\footnotesize
\begin{align*}
&p_{01}= -18763256064938069 - 14792872686537861 t - 4057867882494249 t^2 -\\ 
& 256473877330753 t^3 + 85285038707343 t^4 + 17823637386255 t^5 + \\
& 692357704507 t^6 - 127085232717 t^7 - 14780540079 t^8 -  560972063 t^9 - 31101579 t^{10} - \\
&2597811 t^{11} + 82613 t^{12} - 1323 t^{13} + 57 t^{14} + t^{15},\\
\\
&p_{02}=-2 (-11 + t) (13 + t) (-1674587596733 - 1209041071723 t -\\ 
&   302152954691 t^2 - 21722362269 t^3 + 3705247230 t^4 + 
   789592530 t^5 +\\
&    56234874 t^6 + 3485670 t^7 + 365007 t^8 +    23113 t^9 + 953 t^{10} + 7 t^{11});
\end{align*}

\begin{align*}
&p_{21}=3 (18693496205907553 + 14653197505814513 t + 3981810914612581 t^2 + \\
&   239053845556285 t^3 - 102939835099507 t^4 - 25328343306323 t^5 - \\
&   1853461246959 t^6 + 107242836249 t^7 + 29544857619 t^8 + \\
&   2472362275 t^9 + 146645759 t^{10} + 4570855 t^{11} - 537665 t^{12} + \\
&   13055 t^{13} + 715 t^{14} + 3 t^{15}),\\
\\
&p_{22}=6(-11+t)(-21838212747157 - 17497718905404 t - 5120760089418 t^2 -\\ 
&   504642284812 t^3 + 58729363461 t^4 + 20582639496 t^5 + \\
&   2601849972 t^6 + 260190216 t^7 + 24607269 t^8 + 1916084 t^9 + \\
&   126006 t^{10} + 2820 t^{11} + 11 t^{12});
\end{align*}

\begin{align*}
&p_{41}=-55871225142758423 - 43540681750709607 t - 11717614079999475 t^2 - \\
& 665505935839675 t^3 + 361155464423541 t^4 + 98101500744117 t^5 + \\
& 8901143308057 t^6 - 283763308671 t^7 - 131298459093 t^8 - \\
& 12236812901 t^9 - 634905225 t^{10} + 4984959 t^{11} + 5484599 t^{12} - \\
& 90153 t^{13} - 5469 t^{14} - 5 t^{15},\\
\\
&p_{42}=2 (-11 + t) (65720348410115 + 52809871430820 t + 15313261582566 t^2 + \\
&   1273925478740 t^3 - 273249867411 t^4 - 81619243704 t^5 - \\
&   10990866156 t^6 - 1213822008 t^7 - 112375539 t^8 - 9007084 t^9 - \\
&   714138 t^{10} - 12060 t^{11} + 35 t^{12});
\end{align*}

\begin{align*}
&p_{61}=-(-11 + t)^2 (-153338781731665 - 146672852170183 t - \\
&   57053827731174 t^2 - 10853963116498 t^3 - 364563173779 t^4 + \\
&   353504395539 t^5 + 100621371036 t^6 + 14639637108 t^7 + \\
&   1357232985 t^8 + 82116623 t^9 + 2086442 t^{10} - 122370 t^{11} - \\
&   3157 t^{12} + 5 t^{13}),\\
\\
&p_{62}=-2 (-11 + t) (21975348926173 + 17708836243740 t + 5088010797018 t^2 + \\
&   344542288492 t^3 - 123507628749 t^4 - 33841275336 t^5 - \\
&   4705962132 t^6 - 534799944 t^7 - 46784493 t^8 - 3924308 t^9 - \\
&   371622 t^{10} - 4644 t^{11} + 61 t^{12});
\end{align*}

\begin{align*}
&p_{11}=(-11 + t)^2 (13 + t) (452216947186296794781508787221 + \\
&   777864295951362656514677949267 t + \\
&   607534824156974545331687561835 t^2 + \\
&   288634496196055112144179788837 t^3 + \\
&   94921610624381795338201698150 t^4 + \\
&   23398676158988289024635179242 t^5 + \\
&   4516301732554082896075701906 t^6 + \\
&   670693737388662494219775918 t^7 + 65114539902315732706516167 t^8 + \\
&   225857740971173605970001 t^9 - 1430942150712305946880479 t^{10} - \\
&   342406310658223828546833 t^{11} - 50591510405201802666684 t^{12} - \\
&   5031713916356520207588 t^{13} - 245445214348728475284 t^{14} + \\
&   18666155444026029204 t^{15} + 5518191822155449683 t^{16} + \\
&   667345379819561157 t^{17} + 56579749561878525 t^{18} + \\
&   3879945477820083 t^{19} + 222845520949494 t^{20} + \\
&   10128182363994 t^{21} + 372786923490 t^{22} + 14225061342 t^{23} + \\
&   269308977 t^{24} + 2533431 t^{25} + 12183 t^{26} + 25 t^{27}),\\
\\
&p_{12}=-4 (-11 + t) (13 + t) (31914659841429719306288181169 + \\
&   49622272083198674583533509976 t + \\
&   34282407163818360892131872011 t^2 + \\
&   14089140494250744184954719632 t^3 + \\
&   3939130934203140296373699118 t^4 + \\
&   816295770208132930522370528 t^5 + \\
&   128243247492098759194165618 t^6 + 13314328240910465891671856 t^7 + \\
&   128141672346772035888475 t^8 - 286794878594776251702280 t^9 - \\
&   69849141585341704716287 t^{10} - 10401588386734438049632 t^{11} - \\
&   1051853965102136706572 t^{12} - 44866622100154016896 t^{13} + \\
&   7156442882520008620 t^{14} + 1751453046589473632 t^{15} + \\
&   198334297114129879 t^{16} + 15312353847501512 t^{17} + \\
&   913407761253853 t^{18} + 43228054413008 t^{19} + 1867714065406 t^{20} + \\
&   77464313504 t^{21} + 424793026 t^{22} + 97097840 t^{23} + 2479357 t^{24} + \\
&   20456 t^{25} + 55 t^{26});
\end{align*}

\begin{align*}
&p_{31}=(-11 + t)^2 (13 + t) (4067505216555761454939774893653 + \\
&   6987837642353627260453914645459 t + \\
&   5438613582037781367520396463019 t^2 + \\
&   2560707467679066258265342737957 t^3 + \\
&   824972899826967571738685701350 t^4 + \\
&   195202779101501586572821066218 t^5 + \\
&   35131305236199004729935240210 t^6 + \\
&   4678831943670682549788428718 t^7 + \\
&   361836187431226939045737351 t^8 - 20164101539468052079040559 t^9 - \\
&   13174289998662412173304863 t^{10} - 2713002595179009782526225 t^{11} - \\
&   370593851448133729776060 t^{12} - 34447522144717830870756 t^{13} - \\
&   1447965995249043223956 t^{14} + 171750733434992851092 t^{15} + \\
&   44014595143625089299 t^{16} + 5235351517006620741 t^{17} + \\
&   455693290860748989 t^{18} + 34604751969519795 t^{19} + \\
&   2164988949830262 t^{20} + 86364631885914 t^{21} + 2806443018594 t^{22} + \\
&   191205136350 t^{23} + 833140593 t^{24} - 18388809 t^{25} - 108585 t^{26} +    25 t^{27}),\\
\\
&p_{32}=4 (-11 + t) (13 + t) (-287083238187382918159834379953 - \\
&   445801161270188409230824278008 t - \\
&   306728192538886372185457889323 t^2 - \\
&   124540340456986580024209236752 t^3 - \\
&   33754019007714977189198810542 t^4 - \\
&   6551380444197468706377349664 t^5 - \\
&   918445270855125642258035890 t^6 - 75504009528614463821485232 t^7 + \\
&   3790216746747400729507877 t^8 + 2726753925037134860883880 t^9 + \\
&   554883207716411782569503 t^{10} + 72654419716931233973344 t^{11} + \\
&   6402816324184929410828 t^{12} + 222919538823946521856 t^{13} - \\
&   45365822597641351468 t^{14} - 11428229964701074016 t^{15} - \\
&   1568015533264197079 t^{16} - 148747442138288168 t^{17} - \\
&   8894717908807357 t^{18} - 336458623759952 t^{19} - \\
&   23802738217342 t^{20} - 1386825048800 t^{21} + 41898906878 t^{22} - \\
&   3064255472 t^{23} - 52362109 t^{24} - 143240 t^{25} + 425 t^{26});
\end{align*}

\begin{align*}
&p_{51}=9 (-11 + t)^2 (13 + t) (451683738573925104075654483397 + \\
&   775028304316495850519095857571 t + \\
&   601098380398696980859206188475 t^2 + \\
&   280438400343985227860396599637 t^3 + \\
&   88413801617548106639965501190 t^4 + \\
&   19995978992632195901029965642 t^5 + \\
&   3311133099437108231552511346 t^6 + \\
&   381643355735517005395779022 t^7 + 20320264324033151140865367 t^8 - \\
&   3195804458535397562311231 t^9 - 1143533730000698431197679 t^{10} - \\
&   199161415603964037521793 t^{11} - 24372543185825459073788 t^{12} - \\
&   2046716710595960850596 t^{13} - 54372780278781597012 t^{14} + \\
&   15828476336106803540 t^{15} + 2730101569863744803 t^{16} + \\
&   213625544957043957 t^{17} + 16273483271374381 t^{18} + \\
&   2177665019822083 t^{19} + 170504391614358 t^{20} + \\
&   1407379805882 t^{21} - 64033877566 t^{22} + 34948495614 t^{23} - \\
&   280552415 t^{24} - 2994617 t^{25} - 3033 t^{26} + 9 t^{27}),\\
\\
&p_{52}=36 (-11 + t) (13 + t) (-31880883382920453164668056673 - \\
&   49442138480111650024519040920 t - \\
&   33876857431725516714538231899 t^2 - \\
&   13586828819760368667536109968 t^3 - \\
&   3564076460496389044892230862 t^4 - \\
&   641364123356694605499355488 t^5 - 76820427126193234239727954 t^6 - \\
&   4024537857983776624054576 t^7 + 721610167022215867925013 t^8 + \\
&   250695526123771611854152 t^9 + 42150204570837646184719 t^{10} + \\
&   4607435727083778525024 t^{11} + 285829921804590641228 t^{12} - \\
&   1218353702615312512 t^{13} - 2310870292391561964 t^{14} - \\
&   479326657450856288 t^{15} - 99660901007536295 t^{16} - \\
&   11197664464328328 t^{17} - 270797298387469 t^{18} + \\
&   23020663051568 t^{19} - 2954480013726 t^{20} - 221351668256 t^{21} + \\
&   23088702814 t^{22} - 709577328 t^{23} - 2352877 t^{24} + 17624 t^{25} +   57 t^{26});
\end{align*}

\begin{align*}
&p_{71}=(-11+t)^5(13+t) (-339168361039365519460069487 - 673770165671451476923579920 t - \\
&  624601585018296333977402028 t^2 - 361544324905454191834056976 t^3 - \\
&  147614051919387405233687790 t^4 - 45526903723214507531501616 t^5 - \\
&  11086126516765664128988252 t^6 - 2197894018386795481276848 t^7 - \\
&  362820760120106657958897 t^8 - 50767332334243984565152 t^9 - \\
&  6127600928703138323160 t^{10} - 650077861901582788512 t^{11} - \\
&  61158225664029403460 t^{12} - 5020585598717657952 t^{13} - \\
&  364251213557808792 t^{14} - 29139030508404832 t^{15} - \\
&  3029017347965889 t^{16} - 254576059234128 t^{17} - 7598948126588 t^{18} + \\
&  550936914864 t^{19} + 31892020818 t^{20} - 830153584 t^{21} + \\
&  4681716 t^{22} - 6384 t^{23} + t^{24}),\\
\\
&p_{72}=4(-11+t)^4(13+t) (23939066421797920012252547 + 43604945745949888884259555 t + \\
&  36595686922095182565230747 t^2 + 18937189808814759845550851 t^3 + \\
&  6829585752660013948505437 t^4 + 1838774116182807709355117 t^5 + \\
&  385867254213289689843597 t^6 + 64938062296917527124693 t^7 + \\
&  8939182382326190698494 t^8 + 1017640859072593780990 t^9 + \\
&  95200006411552224398 t^{10} + 7145677316358425566 t^{11} + \\
&  448416529340354362 t^{12} + 29699683025185562 t^{13} + \\
&  1131537557463418 t^{14} - 300028277805366 t^{15} - \\
&  48306393782385 t^{16} + 333976819119 t^{17} + 484360416551 t^{18} + \\
&  17604770671 t^{19} - 1059177895 t^{20} + 10971881 t^{21} - 30551 t^{22} +  17 t^{23}).
\end{align*}

\normalsize
The limiting normalized improper Zolotarev polynomials read:
\begin{equation}\label{limpoly1}
\lim_{t\to-13}Z_{7,t}(x)=-T_6(x)=1-18x^2+48x^4-32x^6,
\end{equation}
\begin{equation}\label{limpoly2}
\lim_{t\to \theta}Z_{7,t}(x)=T_7(y)=-7 y + 56 y^3 - 112 y^5 + 64 y^7
\end{equation}
with $y=\left((1+x)\cos^2(\frac{\pi}{14})-1\right)$, see also \cite[pp. 247-248]{Lebedev94}.
The coefficients of $\widetilde{Z}_{7,t}$ are given by $a_{k,7}(t)=b_{k,7}(t)/b_{7,7}(t), \quad 0\le k\le7$.
\end{theorem}

\begin{example}\label{exarad1a}
We prescribe  $t=t_0=-21 \in I_7$. 
For the septic normalized proper Zolotarev polynomial $Z_{7,t_0}=\widetilde{Z}_{7,t_0}/L_7(t_0)$ we get by insertion of $t_0$ to (\ref{mainthmformula}):
\begin{equation}\label{exa1zpol}
Z_{7,t_0}(x)=\sum_{k=0}^7 b_{k,7}(t_0)x^k
\end{equation}
with
\begin{align}
&b_{0,7}(t_0)=&\frac{5236829509-598896224\sqrt6}{6591796875}=0.5718985919\ldots\\
&b_{1,7}(t_0)\!=\!&\!\!\!\!\frac{-16\sqrt{2 (1174132032293998751 \!+\! 484070220858892414\sqrt6)}}{6591796875}=\nonumber\\
&&-5.2731972200 \ldots\\
&b_{2,7}(t_0)=&\frac{-8(1876889773+537098572\sqrt6)}{2197265625}=-11.6235640503\ldots\\
&b_{3,7}(t_0)\!=\!&\!\!\!\!\!\!\!\frac{16\!\sqrt{2 (56229826406521238759\!+\!22960487256932311726\sqrt6)}}{6591796875}=\nonumber\\
&&36.4042451538 \ldots\\
&b_{4,7}(t_0)=&\frac{8(13748181869+5603740016\sqrt6)}{6591796875}=33.3438497337\ldots\\
&b_{5,7}(t_0)=&\!\!\!\!\!\!\frac{-32\!\sqrt{2 (4907729982259476719\!+\!2003572376153817166\sqrt6)}}{2197265625}=\nonumber\\
&&-64.5264455703 \ldots\\
&b_{6,7}(t_0)=&\frac{-256(299877839+122424446\sqrt6)}{6591796875}=-23.2921842753\ldots\\
&b_{7,7}(t_0)=&\!\!\!\!\!\!\!\!\!\!\!\!\!\!\frac{1024\sqrt{2 (11553696783009431\!+\!4716776960201434\sqrt6)}}{6591796875}=\nonumber\\
&&33.39539763 \ldots\, . \label{exa1b7}
\end{align}
Multiplying these coefficients $b_{k,7}(t_0)$ by $L_7(t_0)=1/b_{7,7}(t_0)$ yields the coefficients $a_{k,7}(t_0)$ of $\widetilde{Z}_{7,t_0}$.
\qed

 Figure \ref{fig:1} displays $\widetilde{Z}_{7,t_0}$ and Figure \ref{fig:2} displays $Z_{7,t_0}$.
 \begin{figure}
\centering
  \begin{minipage}[b]{0.45\textwidth}
    \includegraphics[width=\textwidth]{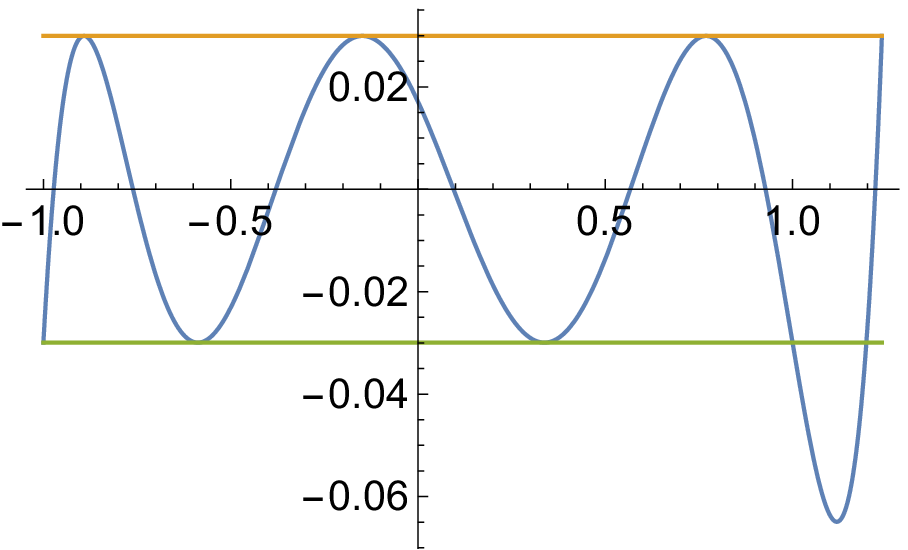}
    \caption{$\widetilde{Z}_{7,t_0=-21}$}
    \label{fig:1}
  \end{minipage} 
  \begin{minipage}[b]{0.45\textwidth}
    \includegraphics[width=\textwidth]{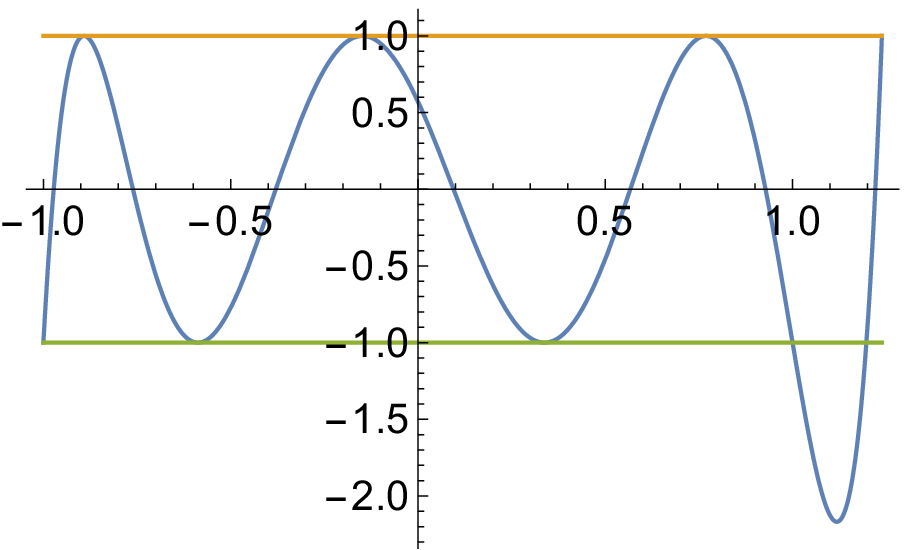}
    \caption{$Z_{7,t_0=-21}$}
    \label{fig:2}
  \end{minipage}
\end{figure}
\end{example}

\section{Derivation by symbolic computation}\label{comppath}
It is known from literature (see 
\cite{Hoeij1}, \cite{Hoeij2}) that algorithms exist for the parametrization of plane algebraic curves of genus $0$ and $1$,  and their implementations are available in computer algebra systems. We will use the \emph{algcurves} package in \emph{Maple}. 
However, as pointed out in \cite[p. 179]{Grasegger17}, the defining reduced relation curve $H_7(\alpha,\beta)=0$, whose points determine the endpoints of the interval 
$[\alpha,\beta]$ (to the right of $\boldsymbol{I}$) on which $Z_{7,t}$ also equioscillates, is  a genus $4$ curve. Therefore the direct parametrization of the curve $H_7=0$, where
(see also Formula (7.9) in \cite{Rack18})
\begin{align}
&H_7(\alpha,\beta)=\\
&(4096  - 10240 \beta^2 + 14080 \beta^4 - 9984 \beta^6 +  1776 \beta^8 +280 \beta^{10} - 7 \beta^{12})+ \nonumber \\
&\alpha (12288 \beta - 13312 \beta^3 + 5632 \beta^5 - 5760 \beta^7 + 1136 \beta^9 + 28 \beta^{11})+ \nonumber\\
& \alpha^2 (-2048 + 25088 \beta^2 - 29952 \beta^4 + 10560 \beta^6 - 3624 \beta^8 + 42 \beta^{10})+ \nonumber\\
& \alpha^3 (-17408 \beta + 33792 \beta^3 - 19328 \beta^5 + 3648 \beta^7 - 484 \beta^9) + \nonumber\\
& \alpha^4 (-8448 - 8448 \beta^2 + 23712 \beta^4 - 7632 \beta^6 + 1311 \beta^8) + \nonumber\\
& \alpha^5 (1536 \beta - 10624 \beta^3 + 11680 \beta^5 - 1800 \beta^7) + \nonumber\\
& \alpha^6 (7424 - 4288 \beta^2 - 3472 \beta^4 + 1260 \beta^6)+ \nonumber\\
& \alpha^7 (4992 \beta - 4032 \beta^3 - 168 \beta^5) + \alpha^8 (-1040 + 1976 \beta^2 - 441 \beta^4) + \nonumber\\
& \alpha^9 (-144 \beta + 364 \beta^3) + \alpha^{10} (184 - 118 \beta^2) + \alpha^{11} 12 \beta+\alpha^{12}, \nonumber
\end{align}
 and thus the proposed way in \cite{Grasegger17} of parametrizing $\widetilde{Z}_{7,t}$ (or $Z_{7,t}$) does not work with the existing implemented algorithms.
This is why in \cite{Grasegger17} the septic case is called a \emph{challenge} which is \emph{subject to further investigation}.  
We remedy this obstacle by the following observation:
\begin{lemma}
Certain 2D projections to the coefficient planes of the algebraic space curve $(a_{0,7}(t),\dots,a_{7,7}(t))\subset \mathbb{R}^8$\ associated to $\widetilde{Z}_{7,t}(x)$, which can be given as a zero-set of  
$P_7(a_j,a_k)$,  have smaller genuses than the defining reduced relation curve $H_7(\alpha,\beta)=0$. In fact, their genuses are determined by the parity of the index-pairs $(j,k),\quad j=0,\dots,5,\, k=j+1,\dots,6$.  
\end{lemma}
\begin{proof}
Using the \emph{algcurves} package, direct exact computation reveals, for example, that
$g(P_7(a_3,a_5)=0)=1<4$. 

To save space, we omit to express the plane projection curve $P_7(a_3,a_5)$ by formula, which would be a quite bulky one.
However, the reader is invited to check the genus computation by recovering $P_7(a_3,a_5)$ from (\ref{simpleradicalpm35}) below, as
\begin{equation*}
P_7(a_3,a_5)\!=\!{\rm res}^{2}_{t}\left({\rm res}^{4}_{o}(r_{31}\!+\!o\, r_{32}\!-\!a_3 q_3,o^2\!\!-\!\omega^2),{\rm res}^{4}_{o}(r_{51}\!+\!o\,r_{52}\!-\!a_5 q_5,o^2\!\!-\!\omega^2)\right),
\end{equation*}
where ${\rm res}^{j}_{x}$ denotes the $j$-th factor of the resultant with respect to the variable $x$. 
\end{proof}

Therefore we use first the known algorithm for the radical parametrizaton of the elliptic curve $P_7(a_3,a_5)=0$.
We compute the Weierstrass normal form and an inverse morphism. After simplification, we so obtain the form given in (\ref{simpleradicalpm35}) below:

\begin{align}\label{simpleradicalpm35}
&a_3=a_{3,7}=a_{3,7}(t)=\frac{r_{31}+\omega r_{32}}{q_{3}} \,\wedge\,a_5=a_{5,7}=a_{5,7}(t)=\frac{r_{51}+\omega r_{52}}{q_{5}},\\
&\hbox{where } r_{3i}=r_{3i}(t), r_{5i}=r_{5i}(t), q_3=q_3(t), \hbox{ and } q_5=q_5(t);\quad i=1,2 \nonumber 
\end{align}
and
\begin{align*}
r_{31}&=-67050634131834282413 - 65805513145120988466 t - &\\
 &33209816302438483773 t^2 - 8207018843674538160 t^3 + &\\
 &61248828418901340 t^4 + 573066725434213512 t^5 + &\\
 &157302381643747500 t^6 + 23037858884482608 t^7 + &\\
 &2281318952792298 t^8 + 137865778703060 t^9 - 7570217151078 t^{10} -&\\ 
 &2900985706704 t^{11} - 249483884148 t^{12} - 3952502904 t^{13} + &\\
 &453522684 t^{14} + 18457680 t^{15} + 1707483 t^{16} - 6834 t^{17} - 5 t^{18},&\\ 
 r_{32}&=5715348771948850560 + 3280169094303556224 t + &\\
 &482405786587260288 t^2 - 71294988953030016 t^3 - &\\
 &37631007796747392 t^4 - 8178434211316608 t^5 - &\\
 &1465523678239872 t^6 - 200841049348992 t^7 - 16338603726720 t^8 - &\\
 &1614798985344 t^9 - 380404637568 t^{10} - 52711585920 t^{11} - &\\
 &3232979328 t^{12} - 83650176 t^{13} - 3459456 t^{14} + 98688 t^{15},&
\end{align*}
\begin{align*}  
q_{3}&=(-11 + t)^2 (13 + t)^{10} (-1763 + 75 t + 111 t^2 + t^3)^2,&\\
r_{51}&=3 (22532652131 - 4558144146 t - 7976565369 t^2 - 1580832984 t^3 -&\\ 
   & 52130250 t^4 + 4221396 t^5 + 476790 t^6 + 87336 t^7 +18279 t^8 -&\\ 
   & 1426 t^9 + 3 t^{10}),&\\
r_{52}&=3 (1455149952 + 337123904 t + 36161984 t^2 + 8501568 t^3 +&\\ 
   & 992704 t^4 + 51904 t^5 - 4800 t^6 + 4544 t^7 - 64 t^8),&\\
q_{5}&=(-11 + t)^2 (13 + t)^5 (-1763 + 75 t + 111 t^2 + t^3).&
\end{align*}
 
By using the polynomial $Q_1(a_5,\alpha)$ which was derived using Groebner basis computation from the Abel-Pell differential equation (\ref{AbelPellEquation})
by coefficient comparison, and
the polynomial $Q_2(a_5,t)$ stemming from the parametrization in (\ref{simpleradicalpm35}), and taking the smaller factor of the resultant 
${\rm res}_{a_5}(Q_1,Q_2)$, we obtain
\begin{align}\label{parpolyalphat}
&Q_3(\alpha,t)=\big(11243366790769 + 1414302826044 t - 1165440897006 t^2 - \\ 
& 225646565396 t^3 + 38213311311 t^4 + 11448834552 t^5 - \nonumber \\
& 152423364 t^6 - 231859656 t^7 - 11635425 t^8 + 1600204 t^9 + \nonumber \\ 
& 138450 t^{10} + 732 t^{11}+t^{12}\big)+ \nonumber \\
& \alpha^2\big(-22569735604130 - 3071056384440 t + 2108952113820 t^2 + \nonumber \\
& 410418431848 t^3 - 47150209566 t^4 - 15418665072 t^5 - \nonumber \\
& 125813496 t^6 + 206967312 t^7 + 7843842 t^8 - 1161176 t^9 - \nonumber\\
& 36324 t^{10} + 2184 t^{11} - 2 t^{12}\big)+ \nonumber\\
& \alpha^4\big(11326363139377 + 1656707880828 t - 943672290990 t^2 - \nonumber \\
& 185096942420 t^3 + 8526067407 t^4 + 3636819576 t^5 + 109344444 t^6 - \nonumber \\
& 23639112 t^7 - 1673505 t^8 + 24844 t^9 + 5394 t^{10} + 156 t^{11} + t^{12}\big). \nonumber
\end{align}  
This is obviously a quadratic expression of $\alpha^2$ and thus $\alpha$ can be expressed in terms of the parameter $t$ by radicals, i.e., $\alpha=\alpha(t)$. 
Similarly,
it turns out that in this way we also obtain a nonsimple radical parametrization for $\beta$ and $s$ in the form of $\beta(t)$ and $s(t)$, see (\ref{pms}) below:
\begin{align}\label{pms}
&\alpha(t)=\sqrt{\frac{pa_1+pa_2\omega}{q_5}},\, \beta(t)=\sqrt{\frac{pb_1+pb_2\omega}{q_5}},\, s(t)=\frac{1}{7}\sqrt{\frac{ps_1+ps_2\omega}{q_5}},\\
& \hbox{where } pa_{i}=pa_{i}(t), pb_{i}=pb_{i}(t), ps_{i}=ps_{i}(t); \quad i=1,2\nonumber
\end{align}
and
\begin{align*}
pa_1=&-78915159455 - 11841667910 t + 6656486445 t^2 + 1445318328 t^3 - \\
& 98097774 t^4 - 45176292 t^5 - 1757742 t^6 + 383160 t^7 + 20493 t^8 - \\
& 1094 t^9 + t^{10},\\
pa_2=&-32 (-237047 - 688048 t - 631196 t^2 - 136576 t^3 + 52510 t^4 +\\ 
&   7760 t^5 - 1292 t^6 - 160 t^7 + t^8),\\
pb_1=&-83889211967 - 23402064230 t - 1593364563 t^2 - 238787016 t^3 - \\
& 84657966 t^4 - 8058276 t^5 + 964434 t^6 + 8760 t^7 + 5229 t^8 + \\
& 154 t^9 + t^{10}, \\
pb_2=&\,\, 55296 (2153 + 4729 t + 3002 t^2 + 434 t^3 + 13 t^4 + 5 t^5),\\
ps_1=&-430297163879 - 140961934838 t - 23363774235 t^2 -\\ 
& 3595306632 t^3 - 594107646 t^4 - 113195268 t^5 - 1046526 t^6 +\\ 
& 1347960 t^7 + 171621 t^8 - 10742 t^9 + 25 t^{10},\\
ps_2=&-96 (-94197721 - 28304354 t - 6941146 t^2 - 1190186 t^3 - 
   19228 t^4 +\\
&    154 t^5 - 310 t^6 - 398 t^7 + 5 t^8).\\
\end{align*}
By using the polynomials from the Groebner basis, formulae can be derived for $a_{1,7}(t)$ (similar to $a_{3,7}(t)$ and $a_{5,7}(t)$ in (\ref{simpleradicalpm35})) and for the even-indexed coefficents $a_{0,7}(t)$, $a_{2,7}(t)$, $a_{4,7}(t)$ and for $L_7(t)$ (similar to $s(t)$ in (\ref{pms})), so that we can determine, for $n=7$, 
the monic polynomial in (\ref{mzolpargen}), with $a_{6,7}(t)=-7s(t)$ and $a_{7,7}(t)=1$.
But we omit these coefficient-formulae in order to save space, since 
they can be recovered from Theorem \ref{mainthm}. 
In fact, since $-L_7(t)=\widetilde{Z}_{7,t}(1)=1+\sum_{k=0}^6 a_k(t)$, it follows 
from $Z_{7,t}(x)=\widetilde{Z}_{7,t}/L_7(t)$ that for the coefficients of $Z_{7,t}$ there holds $b_{k,7}(t)=a_{k,7}(t) / L_7(t)$, in particular
$b_{7,7}(t)=1/L_7(t)$.
 

In this way we have 
obtained our radical parametrization of the septic normalized proper  
Zolotarev polynomials  $Z_{7,t}$ in Theorem \ref{mainthm}.

As for the parameter interval $I_7$, we first determine the possible real parameter values $t$ for which 
\begin{align}
&a_{3,7}(t)={\rm Root[}25621 - 21445 x - 3609 x^2 + x^3,2{\rm]}\,\,\wedge \\ 
&a_{5,7}(t)={\rm Root[}2443 + 651 x - 343 x^2 + x^3,1{\rm]} \nonumber
\end{align} 
holds, where the right hand sides of the above equations are the coefficients of the monic version of the limiting polynomial in (\ref{limpoly2}). We obtain the unique solution $\theta={\rm Root[}-1847-93x+27x^2+x^3,1{\rm]}$. Second, using the expected range  $(a_{3,7}(\theta),19/16)\times (-2,a_{5,7}(\theta))$ for $(a_{3,7}(t),a_{5,7}(t))$,
we conclude that the only appropriate choice for $I_7$ is $(\theta,-13)$ as given in Theorem \ref{mainthm}.  \qed
 

 


\begin{example}\label{exarad1}
We again prescribe  $t=t_0=-21 \in I_7$, which corresponds to prescribe $s(t_0)=s_0\in\left(\tan^2(\frac{\pi}{14}),\infty\right)=(0.0520950836\dots,\infty)$,
and get for $L_7(t_0)$, $s(t_0)$, $\alpha$, $\beta$, $\gamma=\frac{\alpha+\beta}{2}-s(t_0)\,$
the concrete values
\begin{align}\label{exa1L}
&L_7(t_0)=&
\frac{6591796875}{1024\sqrt{2 (11553696783009431\!+\!4716776960201434\sqrt6)}}=\nonumber\\
&&0.0299442459 \ldots
\end{align} 
\begin{equation}\label{exa1s}
s(t_0)=\frac{1}{28}\sqrt{\frac{-24764015 + 10110318\sqrt6}{142}}=0.0996381277 \ldots
\end{equation}
\begin{align}
&\alpha(t_0)=\alpha=\frac{1}{4}\sqrt{\frac{-1816103+742750\sqrt6}{142}}=1.1970302256 \ldots \label{exa1alpha}\\
&\beta(t_0)=\beta=\frac{1}{4}\sqrt{\frac{-330503+136350\sqrt6}{142}}=1.2384903969 \ldots \label{exa1beta} \\
&\gamma(t_0)=\gamma=\frac{1}{7}\sqrt{\frac{-2665+1138\sqrt6}{2}}=1.1181221834 \ldots\,.\, \qed  \label{exa1gamma}
\end{align}

\end{example}
\begin{example}
The goal is to solve ZFP (\ref{genZFP}) for $n=7$ and for $s=s_0=2>\tan^2(\frac{\pi}{14})$,  
say. To this end, we replace in (\ref{pms}) the left-hand side $s(t)$ by $2$ and solve the
corresponding equation for the variable $t$, where the unique solution from $I_7$ reads
$t=t_0={\rm Root[}u_{12}(x),2{\rm]}=-13.0305732483\ldots\,$, with
\begin{align}
&u_{12}(x)=\\
&45678110765558881 + 6234622186059196 x - 4142454922929870 x^2 - \nonumber\\
& 804497641938708 x^3 + 33981103268895 x^4 + 15022950168312 x^5 + \nonumber\\
& 496259775228 x^6 - 87343156680 x^7 - 6119479377 x^8 - 32085428 x^9 + \nonumber \\
& 16838386 x^{10} + 1138140 x^{11} + 3249 x^{12}.\nonumber
 \end{align}
 Then we insert this $t_0$ into $\frac{Z_{7,t}(x)}{b_{7,7}(t)}$ to obtain
 \begin{align}
 &\widetilde{Z}_{7,t_0}(x)=\widetilde{Z}_{7,s_0}^{*}(x)=\\
 & 0.4369440905\ldots+(-0.1873410678\ldots)x+(-7.8705484829\ldots)x^2+ \nonumber\\
 &(1.1870230011\ldots)x^3+(20.9955470665\ldots)x^4+(-1.9996819333\ldots)x^5+ \nonumber\\
 &(-14)x^6+x^7.\nonumber
\end{align}
This means that $\min_{(a_{0,7},\cdots,a_{5,7)}\in \mathbb{R}^{6}}||P_7||_{\infty}=||\widetilde{Z}_{7,t_0}||_{\infty}=||\widetilde{Z}_{7,s_0}^{*}||_{\infty}=L_7(t_0)=L_7(s_0)$
where $P_7(x)=\sum_{k=0}^5 a_{k,7}x^k+(-14)x^6+x^7$, i.e. its first two leading coefficients are prescribed. Evaluating
$L_7(t)=\frac{1}{b_{7,7}(t)}$ at $t=t_0$ (see (\ref{mainthmformula})), gives $L_7(t_0)=0.4380573257\ldots$ \,.\qed 
\end{example}

\section{A note on the equations of Abel-Pell and Peherstorfer-Schiefermayr for $n=7$}

The Abel-Pell differential equation for $Z_{7,t}$ reads, see (\ref{AbelPellEquation}), (\ref{mainthmformula}), (\ref{pms}),
\begin{equation}\label{AbelPellEquationn7}
\frac{(1-x^2)(x-\alpha(t))(x-\beta(t))(Z_{7,t}\,'(x))^2}{49(x-\gamma(t))^2}+(Z_{7,t}(x))^2=1.
\end{equation}
Since all terms are defined for $n=7$, a simplification with \emph{Maple} or \emph{Mathematica} will 
confirm that (\ref{AbelPellEquationn7}) holds true. For the special parameter value $t=t_0=-21$
a verification of (\ref{AbelPellEquationn7}) can be carried out by using (\ref{exa1zpol})--(\ref{exa1b7}) and (\ref{exa1s})--(\ref{exa1gamma}).

The Peherstorfer-Schiefermayr system of nonlinear equations (\ref{PSE})--(\ref{PSE2}) reads, for $n=7$,
\begin{equation}\label{PSE71}
\alpha(t)+\beta(t)+2(z_1+z_2+z_3+z_4+z_5)-14s(t)=0
\end{equation}
\begin{equation}\label{PSE72}
(-1)^k+2(-z_1^k+z_2^k-z_3^k+z_4^k-z_5^k)+1+(\alpha(t))^k-(\beta(t))^k=0\,(k=1,\dots,6).
\end{equation}
The equioscillation points $z_2$ and $z_4$, where $Z_{7,t}(z_2)=Z_{7,t}(z_4)=-1$ and $Z_{7,t}'(z_2)=Z_{7,t}'(z_4)=0$ holds,
can be given in an explicit form as follows:
\begin{align}
&z_2=z_2(t)=\\
&\!-\!\sqrt{\frac{(t\!-\!11) p_9 \!-\! 16 \omega (1 + t)^2 p_6 \!+\!8  (1 + t)^2 (13 + t)\sqrt{3(11\!-\!t) (p_{12} \!-\! 2 \omega p_{10}) }}{q_5}}, \nonumber
\end{align}
\begin{align}
&z_4=z_4(t)=\\
&\sqrt{\frac{(t\!-\!11) p_9 \!-\! 16 \omega(1 + t)^2 p_6  \!-\!8  (1 + t)^2 (13 + t)\sqrt{3(11\!-\!t) (p_{12} \!-\! 2 \omega p_{10}) }}{q_5}}, \nonumber
\end{align}
where 
\begin{align*}
&p_9=p_9(t)=\\
&7458346453 + 2414057145 t + 80738436 t^2 - 35137404 t^3 +  \\
& 943590 t^4 + 1460238 t^5 + 136212 t^6 + 4020 t^7 - 339 t^8 + t^9, \\
&p_6=p_6(t)=\\
&-4757750 - 991641 t + 35097 t^2 + 1478 t^3 - 84 t^4 - 93 t^5 + t^6,\\
&p_{12}=p_{12}(t)=\\
&9627927080284 + 4974116032425 t + 1038680780799 t^2 + \\
& 126082292719 t^3 + 6771588669 t^4 + 996190362 t^5 + 407573430 t^6 + \\
& 26107902 t^7 - 4154346 t^8 - 474323 t^9 + 59547 t^{10} -  861 t^{11} + t^{12},\\
&p_{10}=p_{10}(t)=\\
&122595778519 + 52782140344 t + 8848501713 t^2 + 1458683184 t^3 + \\
& 201606006 t^4 + 6367392 t^5 - 2103342 t^6 - 42288 t^7 + 68355 t^8 - \\
& 2648 t^9 + 13 t^{10}. 
\end{align*}
The equioscillation points $z_1,z_3$ and $z_5$ where $Z_{7,t}(z_1)=Z_{7,t}(z_3)=Z_{7,t}(z_5)=1$ and $Z_{7,t}'(z_1)=Z_{7,t}'(z_3)=Z_{7,t}'(z_5)=0$ holds,
can be represented in a similar but more bulky form (as solutions of a cubic polynomial equation with \emph{casus irreducibilis}) and are omitted.

For the special parameter value $t=t_0=-21$ a verification of (\ref{PSE71}), (\ref{PSE72}) can be carried out by using $s(t_0)$, $\alpha$ and $\beta$
from (\ref{exa1s})--(\ref{exa1beta}) and the following identities:
\begin{align}
&z_1=-\sqrt{\tau_3}=-0.8914485687\ldots, \hbox{ where } \tau_3={\rm Root[}u_3(x),3{\rm]} \hbox{ and}\\
\nonumber \\
&u_3(x)=-150\sqrt{6}(568552997 + 66091829356 x + 136570772 x^2)+ \nonumber\\
&208899235967 + 24283711656924 x + 50114778876 x^2 + 45812608 x^3 \nonumber\\
\nonumber\\
&z_2=\!\!-\frac{1}{8}\sqrt{\frac{1}{71}(\!-1744831 \!\!+\!\! 712750\sqrt6\!\!+\!\!25\sqrt{2104102185\!-\! 858995940\sqrt6}}\!=\\
&\hspace{9cm}-0.5873784916 \ldots\nonumber\\
\nonumber \\
&z_3=-\sqrt{\tau_1}=-0.1483533115\ldots, \hbox{ where } \tau_1={\rm Root[}u_3(x),1{\rm]}\\
\nonumber \\
&z_4=\!\!\frac{1}{8}\sqrt{\frac{1}{71}(\!-1744831 \!+\! 712750\sqrt6\!-\!25\sqrt{2104102185\!-\! 858995940\sqrt6}}=\\
&\hspace{9.6cm} {0.3375968260 \ldots}\nonumber\\
\nonumber \\
&z_5=\sqrt{\tau_2}=0.7692901289\ldots, \hbox{ where } \tau_2={\rm Root[}u_3(x),2{\rm]}.
\end{align}
\section{Concluding remarks}

\begin{remark}
The obtained septic radical parametrization in Theorem \ref{mainthm} which is similar to, but more complicated than the one for $n=5$ in \cite{Grasegger17} 
and for $n=6$ in \cite{RackVajda19}, may possibly be further simplified. To find a re-parametrization for $n=7$ is subject to further investigations. 
\end{remark}
\begin{remark}
A systematic computation of the genuses of all possible plane projection curves $P_{n}(a_j,a_k)=0$ of the proper Zolotarev polynomials $\widetilde{Z}_{n,t}$ and $Z_{n,t}$ for $n>4$ would shed a light on the structure of the associated space curves.  
\end{remark}
\begin{remark}
The only parameter constellation $t\in I_7$ where $\widetilde{Z}_{7,t}=Z_{7,t}$ holds is
$t=t_0=-13.0058608055\ldots$  (which is the root of an integer 
polynomial of degree 56) since then $b_{7,7}(t)=1$ holds.
\end{remark}


%




\begin{thebibliography}{00}




\bibitem{Achieser98} 
{\sc N.I. Achieser}, {\it Function theory according to Chebyshev}. In: Mathematics of the 
     19th century, Vol. 3 (A.N. Kolmogorov et al. (Eds.)), Birkh\"auser, Basel, 1998, 
     pp. 1-81 (Russian 1987).

\bibitem{Achieser03} 
{\sc N.I. Achieser}, {\it Theory of Approximation}, Dover Publications, Mineola (NY), USA, 2003 (Russian 1947). DOI: 10.2307/3612294 



\bibitem{Bernstein1952}
{\sc S.N. Bernstein}, {\it Collected Works, Vol. I. Constructive Theory of Functions (1905-
     1930)} (Russian), Akad. Nauk SSSR, Moscow, Russia, 1952. 



\bibitem{Bogatyrev02}
{\sc A. B. Bogatyrev}, {\it Effective approach to least deviation problems}, Sb. Mat. {\bf 193} (2002), pp. 21-40. 
DOI: 10.1070/SM2002v193n12ABEH000698


\bibitem{Carlson83} 
{\sc B.C. Carlson, J. Todd}, {\it Zolotarev's first problem - the best approximation by polynomials 
of degree $\le n-2$ to $x^n-n\sigma x^{n-1}$} in $[-1,1]$, Aeq. Math. {\bf 26} (1983), pp. 1-33. DOI: 10.1007/bf02189661

\bibitem{Chebyshev} 
{\sc P.L. Chebyshev}, {\it Th\'eorie des m\'ecanismes connus sous le nom de  
     parall\'elogrammes}, Mem. Acad. Sci. St. Petersburg {\bf 7} (1854), pp. 539-568. 
URL {\tt http://www.math.technion.ac.il/hat/fpapers/cheb11.pdf}


\bibitem{Chen}
{\sc X. Chen, Th.W. Parks}, Analytic design of optimal FIR narrow-band filters using Zolotarev polynomials, 
 IEEE Transactions on Circuits and Systems CAS-33 (1986), pp. 1065-1071.

\bibitem{Collins96} 
{\sc G.E. Collins}, {\it Application of quantifier elimination to Solotareff's approximation problem}. 
In: Stability Theory (Hurwitz Centenary Conference, Ascona, Switzerland, 1995, R. Jeltsch et al. (Eds.)), Birkh\"auser, Basel, 
ISNM {\bf 121} (1996), pp. 181-190. DOI: 10.1007/978-3-0348-9208-7\_19

\bibitem{Erdos42}
{\sc P. Erd\H{o}s, G. Szeg\H{o}}, {\it On a problem of I. Schur}, Ann. Math. {\bf 43} (1942), pp. 451-470. DOI: 10.2307/1968803 






\bibitem{Grasegger17}
{\sc G. Grasegger, N.Th. Vo}, {\it An algebraic-geometric method for computing  
       Zolotarev polynomials}.
In: Proceedings International Symposium on Symbolic and 
       Algebraic Computation (ISSAC '17, Kaiserslautern, Germany, M. Burr (Ed.)), 
       ACM, New York, 2017, pp. 173-180. DOI: 10.1145/3087604.3087613


\bibitem{Hoeij1}
{\sc M. van Hoeij},
{\it Rational parametrizations of algebraic curves using a canonical divisor}, J. Symb. Comput. {\bf 23} (1997), pp. 209-227.

\bibitem{Hoeij2}
{\sc M. van Hoeij}, {\it An algorithm for computing the Weierstrass normal form}. In: 
Proceedings International Symposium on Symbolic and Algebraic Computation (ISSAC'95, Montreal, Canada, A. H. M Levelt (Ed.)),
ACM, New York, 1995,  pp. 90-95. DOI: 10.1145/220346.220358


\bibitem{Kaltofen} 
{\sc E. Kaltofen}, {\it Challenges of symbolic computation: My favorite open problems}, 
       J. Symb. Comput. {\bf 29} (2000), pp. 891-919. DOI: 10.1006/jsco.2000.0370

\bibitem{Lazard} 
{\sc D. Lazard}, {\it Solving Kaltofen's challenge on Zolotarev's approximation}.
In: Proceedings International Symposium on Symbolic and Algebraic Computation 
       (ISSAC '06, Genoa, Italy, B. Trager (Ed.)), ACM, New York, 2006, pp. 196-203. DOI: 10.1145/1145768.1145803

\bibitem{Lebedev94}
{\sc V.I. Lebedev}, {\it Zolotarev polynomials and extremum problems}, Russ. J. Numer. 
       Anal. Math. Model. {\bf 9} (1994), pp. 231-263. DOI: 10.1515/rnam.1994.9.3.231 


\bibitem{Malyshev02} 
{\sc V.A. Malyshev}, {\it The Abel equation}, St. Petersburg Math. J. {\bf 13} (2002), pp. 893-938 (Russian 2001).



\bibitem{maple}
{\sc Maplesoft}, {\it Maple 2019}, Maplesoft, a division of Waterloo Maple Inc., Waterloo, Ontario, Canada.

\bibitem{AAMarkov1889}
{\sc A.A. Markov}, {\it On a question of D.I. Mendeleev} (Russian), Zapiski Imper. Akad. 
       Nauk St. Petersburg {\bf 62} (1889), pp. 1-24. URL {\tt http://www.math.technion.ac.il/hat/fpapers/mar1.pdf}
   

\bibitem{AAMarkov1906}
{\sc A.A. Markov}, {\it Lectures on functions of minimal deviation from zero} (Russian), 
       1906. In: Selected Works: Continued fractions and the theory of functions deviating 
       least from zero, OGIZ, Moscow-Leningrad, 1948, pp. 244-291.


\bibitem{VAMarkov}
{\sc V.A. Markov}, {\it On functions deviating least from zero on a given interval} 
       (Russian), Izdat. Akad. Nauk St. Petersburg 1892. 
URL {\tt  http://www.math.technion.ac.il/hat/fpapers/vmar.pdf}


\bibitem{Milovanovic}
{\sc G.V. Milovanovi\'c, D.S. Mitrinovi\'c, Th.M. Rassias}, {\it Topics in Polynomials: 
       Extremal Problems, Inequalities, Zeros}, World Scientific, Singapore, 1994. DOI: 10.1142/9789814360463

\bibitem{Paszkowski62}
{\sc S. Paszkowski}, {\it The Theory of Uniform Approximation I. Non-asymptotic 
       Theoretical Problems}, Rozprawy Matematyczne XXVI, PWN, Warsaw, Poland, 1962.



\bibitem{Peherstorfer91}
{\sc F. Peherstorfer}, {\it On the connection of Posse's $L_1$- and Zolotarev's maximum-norm problem}, 
J. Approx. Theory {\bf 66} (1991), pp. 288-301. DOI: 10.1016/0021-9045(91)90032-6


\bibitem{Peherstorfer99} 
{\sc F. Peherstorfer, K. Schiefermayr}, {\it Description of extremal polynomials on 
       several intervals and their computation I, II}, Acta Math. Hungar. {\bf 83} (1999), pp. 27-58, pp. 59-83.
DOI: 10.1023/a:1006607401740 and DOI: 10.1023/a:1006659402649


\bibitem{Peherstorfer04} 
{\sc F. Peherstorfer, K. Schiefermayr}, {\it Description of inverse polynomial    
       images which consist of two Jordan arcs with the help of Jacobi's elliptic functions}, 
       Comput. Methods Funct. Theory {\bf 4} (2004), pp. 355-390. DOI: 10.1007/bf03321075


\bibitem{Peherstorfer06}
{\sc F. Peherstorfer}, {\it Asymptotic representation of Zolotarev polynomials}, J. London 
     Math. Soc. {\bf 74} (2006), pp. 143-153. DOI: 10.1112/s0024610706022885


\bibitem{Rack89} 
{\sc H.-J. Rack}, {\it On polynomials with largest coefficient sums}, J. Approx. Theory {\bf 56} 
       (1989), pp. 348-359. DOI: 10.1016/0021-9045(89)90124-x

\bibitem{Rack17a} 
{\sc H.-J. Rack}, {\it The first Zolotarev case in the Erd\H{o}s-Szeg\H{o} solution to a Markov-type 
       extremal problem of Schur}, Stud. Univ. Babes-Bolyai Math. {\bf 62} (2017), pp. 151-162. DOI: 10.24193/subbmath.2017.2.02

\bibitem{Rack17b} 
{\sc H.-J. Rack}, {\it The second Zolotarev case in the Erd\H{o}s-Szeg\H{o} solution to a Markov-
       type extremal problem of Schur}, J. Numer. Anal. Approx. Theory {\bf 46} (2017), pp. 54-77.

\bibitem{Rack18}
{\sc H.-J. Rack, R. Vajda}, {\it Explicit algebraic solution of Zolotarev's first problem for low-degree polynomials}, 
J. Numer. Anal. Approx. Theory {\bf 48} (2019) pp. 175-201.





\bibitem{RackVajda19}
{\sc H.-J. Rack, R. Vajda}, {\it An explicit univariate and radical parametrization of the sextic proper Zolotarev polynomials in power form}, Dolomites Res. Notes Approx. 
{\bf 12} (2019), pp. 43-50. DOI: 10.14658/pubj-drna-2019-1-4


\bibitem{Rivlin75}
{\sc Th.J. Rivlin}, {\it Optimally stable Lagrangian numerical differentiation}, SIAM J. Numer. Anal. {\bf 12} (1975), pp. 712-725. 
DOI: 10.1137/0712053


\bibitem{Rivlin74} 
{\sc Th.J. Rivlin}, {\it Chebyshev Polynomials}, 2nd edition, J. Wiley \& Sons, New York (NY), USA, 1990. DOI: 10.2307/2153227 





\bibitem{Shadrin04}
{\sc A. Shadrin}, {\it Twelve proofs of the Markov inequality}. 
In: Approximation Theory - A volume dedicated to B. Bojanov (D.K. Dimitrov et al. (Eds.)), M. Drinov Acad. 
       Publ. House, Sofia, 2004, pp. 233-298. 
URL {\tt http://www.damtp.cam.ac.uk/user/na/people/} {\tt Alexei/papers/markov.pdf}

\bibitem{Shadrin14}
{\sc A. Shadrin}, {\it The Landau-Kolmogorov inequality revisited}, Discrete Contin. Dyn. 
       Syst. {\bf 34} (2014), pp. 1183-1210. DOI: 10.3934/dcds.2014.34.1183




\bibitem{Tikhomirov}
{\sc V.M. Tikhomirov}, {\it II. Approximation Theory}. In: Analysis II (R.V. Gramkelidze (Ed.)), Encyclopaedia 
of Mathematical Sciences, Vol. 14, Springer, New York (NY), USA, 1990, pp. 93-243 (Russian 1987). 
DOI: 10.1007/978-3-642-61267-1\_2 

\bibitem{Todd84}
{\sc J. Todd}, {\it Applications of transformation theory: A legacy from Zolotarev 
       (1847-1878)}, In: Approximation Theory and Spline Functions, Proceedings NATO 
       Advanced Study Institute (ASIC 136, St. John's, Canada, 1983, S.P. Singh et al. 
       (Eds.)), 1984, pp. 207-245. DOI: 10.1007/978-94-009-6466-2\_11 


\bibitem{Voronovskaja70} 
{\sc E.V. Voronovskaja}, {\it The Functional Method and its Applications}, Translations 
       of Mathematical Monographs, Vol. 28, AMS, Providence (RI), USA, 1970 (Russian 
       1963). DOI: 10.1090/mmono/028 


\bibitem{MMA1}
{\sc Wolfram Research, Inc.}, {\it Mathematica}, Version 12.0, Champaign (IL), USA, 2019.

\bibitem{Zolotarev1868}{\sc E.I. Zolotarev}, {\it On a problem of least values, Dissertation pro venia legendi}
  (Russian), lithographed paper, St. Petersburg (1868), Collected Works, Vol. 2, 
       Leningrad 1932, pp. 130-166.

\bibitem{Zolotarev1877} {\sc E.I. Zolotarev}, {\it Applications of elliptic functions to problems on functions 
       deviating least or most from zero} (Russian), Zapiski Imper. Akad. Nauk St. 
       Petersburg {\bf 30} (1877), Collected Works, Vol. 2, Leningrad, 1932, pp. 1-59. 
URL {\tt  http://www.math.technion.ac.il/hat/fpapers/zolo1.pdf}





\end{thebibliography}


\end{document}